\documentclass[12pt, reqno ]{amsart}
\usepackage{amsmath,amssymb,amsbsy,amsfonts,amsthm,latexsym, amsopn,amstext,amsxtra,euscript,amscd,mathrsfs,color,bm, cite,multirow,tabularx}
\usepackage{todonotes,dirtytalk}   
\usepackage{url}
\usepackage[colorlinks,linkcolor=blue,anchorcolor=blue,citecolor=blue]{hyperref}
\usepackage{color}
\usepackage{comment}
\usepackage{dirtytalk}
\usepackage{lipsum}

\begin{document}
\newtheorem{problem}{Problem}
\newtheorem{theorem}{Theorem}
\newtheorem{lemma}[theorem]{Lemma}
\newtheorem{crit}[theorem]{Criterion}
\newtheorem{claim}[theorem]{Claim}
\newtheorem{cor}[theorem]{Corollary}
\newtheorem{prop}[theorem]{Proposition}
\newtheorem{definition}{Definition}
\newtheorem{question}[theorem]{Question}
\newtheorem{rem}[theorem]{Remark}
\newtheorem{Note}[theorem]{Notation}


\def\cA{{\mathcal A}}
\def\cB{{\mathcal B}}
\def\cC{{\mathcal C}}
\def\cD{{\mathcal D}}
\def\cE{{\mathcal E}}
\def\cF{{\mathcal F}}
\def\cG{{\mathcal G}}
\def\cH{{\mathcal H}}
\def\cI{{\mathcal I}}
\def\cJ{{\mathcal J}}
\def\cK{{\mathcal K}}
\def\cL{{\mathcal L}}
\def\cM{{\mathcal M}}
\def\cN{{\mathcal N}}
\def\cO{{\mathcal O}}
\def\cP{{\mathcal P}}
\def\cQ{{\mathcal Q}}
\def\cR{{\mathcal R}}
\def\cS{{\mathcal S}}
\def\cT{{\mathcal T}}
\def\cU{{\mathcal U}}
\def\cV{{\mathcal V}}
\def\cW{{\mathcal W}}
\def\cX{{\mathcal X}}
\def\cY{{\mathcal Y}}
\def\cZ{{\mathcal Z}}

\def\A{{\mathbb A}}
\def\B{{\mathbb B}}
\def\C{{\mathbb C}}
\def\D{{\mathbb D}}
\def\E{{\mathbb E}}
\def\F{{\mathbb F}}
\def\G{{\mathbb G}}
\def\I{{\mathbb I}}
\def\J{{\mathbb J}}
\def\K{{\mathbb K}}
\def\L{{\mathbb L}}
\def\M{{\mathbb M}}
\def\N{{\mathbb N}}
\def\O{{\mathbb O}}
\def\P{{\mathbb P}}
\def\Q{{\mathbb Q}}
\def\R{{\mathbb R}}
\def\S{{\mathbb S}}
\def\T{{\mathbb T}}
\def\U{{\mathbb U}}
\def\V{{\mathbb V}}
\def\W{{\mathbb W}}
\def\X{{\mathbb X}}
\def\Y{{\mathbb Y}}
\def\Z{{\mathbb Z}}

\def\ep{{\mathbf{e}}_p}
\def\eq{{\mathbf{e}}_q}
\def\cal#1{\mathcal{#1}}

\def\scr{\scriptstyle}
\def\\{\cr}
\def\({\left(}
\def\){\right)}
\def\[{\left[}
\def\]{\right]}
\def\<{\langle}
\def\>{\rangle}
\def\fl#1{\left\lfloor#1\right\rfloor}
\def\rf#1{\left\lceil#1\right\rceil}
\def\le{\leqslant}
\def\ge{\geqslant}
\def\eps{\varepsilon}
\def\mand{\qquad\mbox{and}\qquad}

\def\sssum{\mathop{\sum\ \sum\ \sum}}
\def\ssum{\mathop{\sum\, \sum}}
\def\ssumw{\mathop{\sum\qquad \sum}}

\def\vec#1{\mathbf{#1}}
\def\inv#1{\overline{#1}}
\def\num#1{\mathrm{num}(#1)}
\def\dist{\mathrm{dist}}

\def\fA{{\mathfrak A}}
\def\fB{{\mathfrak B}}
\def\fC{{\mathfrak C}}
\def\fU{{\mathfrak U}}
\def\fV{{\mathfrak V}}

\newcommand{\bflambda}{{\boldsymbol{\lambda}}}
\newcommand{\bfxi}{{\boldsymbol{\xi}}}
\newcommand{\bfrho}{{\boldsymbol{\rho}}}
\newcommand{\bfnu}{{\boldsymbol{\nu}}}

\def\GL{\mathrm{GL}}
\def\SL{\mathrm{SL}}

\def\Hba{\overline{\cH}_{a,m}}
\def\Hta{\widetilde{\cH}_{a,m}}
\def\Hb1{\overline{\cH}_{m}}
\def\Ht1{\widetilde{\cH}_{m}}

\def\flp#1{{\left\langle#1\right\rangle}_p}
\def\flm#1{{\left\langle#1\right\rangle}_m}

\def\Zm{\Z/m\Z}

\def\Err{{\mathbf{E}}}
\def\O{\mathcal{O}}

\def\cc#1{\textcolor{red}{#1}}
\newcommand{\commT}[2][]{\todo[#1,color=green!60]{Tim: #2}}
\newcommand{\commB}[2][]{\todo[#1,color=red!60]{Bryce: #2}}

\newcommand{\comm}[1]{\marginpar{%
\vskip-\baselineskip 
\raggedright\footnotesize
\itshape\hrule\smallskip#1\par\smallskip\hrule}}
\newcolumntype{L}{>{\raggedright\arraybackslash}X}

\def\xxx{\vskip5pt\hrule\vskip5pt}

\def\dmod#1{\,\left(\textnormal{mod }{#1}\right)}


\title{A note on least primitive roots}

\author[G.K. Bagger]{Gustav Kj\ae rbye Bagger}
\address{School of Science, The University of New South Wales Canberra, Australia}
\email{g.bagger@unsw.edu.au}
 
\date{\today}
\pagenumbering{arabic}


\begin{abstract}
The least primitive root $g(p)$ modulo a prime $p$ is conjectured by Grosswald to satisfy $g(p)<\sqrt{p}-2$ for any $p>409$. We make progress towards the conjecture by proving $g(p)<\sqrt{p}-2$ when $p > 1.77 \times 10^{54}$. 
\end{abstract}

\maketitle
\let\thefootnote\relax
\footnote{\textit{Affiliation}: School of Science, The University of New South Wales Canberra, Australia.}
\footnote{\textit{Key phrases}: Primitive elements, Grosswald's conjecture, Prime sieves}
\section{Introduction}
An integer is a primitive root modulo $n$ if it is a generator for the multiplicative group $\left(\Z/n\Z\right)^\times$. If the modulus is a prime $p$, there are $\phi(p-1)$ such roots modulo $p$, but locating these primitive roots is non-trivial in general. If we denote by $g(p)$ the least primitive root modulo $p$, then Burgess~\cite{Burgess1962} gives an asymptotic bound
$$ g(p) \ll p^{1/4+\epsilon}. $$
Indeed, Grosswald~\cite{Grosswald} conjectured 
\begin{equation} \label{Grosswald}
g(p) < \sqrt{p} - 2 \quad \text{ for all } p>409. 
\end{equation} 
Note that the cutoff $p>409$ arises since $g(409)=21>\sqrt{409}-2$ and we will assume $p>409$ throughout, unless otherwise specified. If one is willing to assume the Generalised Riemann Hypothesis, then~\eqref{Grosswald} is resolved completely (see~\cite{McGownTrevinoTrudgian2017}). Unconditionally, the conjecture holds whenever $p \notin [10^{16},10^{56}]$ with the lower and upper limits due to McGown and Sorensen~\cite{McGownSorensen2025} and McGown and Trudgian~\cite{McGownTrudgian2020}, respectively. We improve upon this upper limit and suggest a method for further refinements.
\begin{theorem} \label{T: Gross int}
$$g(p) < \sqrt{p} - 2 \quad \text{\normalfont for all } p>1.77 \times 10^{54}. $$
\end{theorem}
\section{The sieving criterion}
We begin by giving a slight refinement of \cite[Theorem 2]{McGownTrudgian2020} for $r=2$ by optimising the leading constant and working over the interval $10^{47}<p<10^{56}$. The lower limit is chosen in order to guarantee $g(p)<p^{0.56}$, this is Theorem 5 in~\cite{McGownSorensen2025}.
\begin{prop} \label{P:C=1.4195}
Suppose $10^{47}<p<10^{56}$. Let $e$ be an even divisor of $p-1$, with $p_1\dots p_s \mid p-1$ but $p_i \nmid e$. If $\delta := 1- \sum_{i=1}^sp_i^{-1}>0$, then
$$ g(p) < 2.8390 \left(\frac{2\delta + s - 1}{\delta}2^{\omega-s}\right)^2 p^{3/8}. $$
\end{prop}
We will prove Proposition~\ref{P:C=1.4195} by invoking the following result.
\begin{theorem} \label{T: M and T}
Let $p$ be an odd prime and $e$ be an even divisor of $p-1$. Consider the divisors $p_1\dots p_s \mid p-1$, where $p_i \nmid e$. Let $\delta := 1- \sum_{i=1}^sp_i^{-1}$ and $h>1$ an integer. Choose $H>0$ and set $X:=H/h$. Suppose $\delta>0$, $X\geq 2$ and $2H^2 < hp$. If
\begin{equation} \label{E: T3} 
\frac{\pi^2}{6}\frac{B(X)^3}{A(X)^4}\left(\frac{2\delta + s-1}{\delta}2^{\omega-s}\right)^4hp^{1/2}\left(3+2\frac{p^{1/2}}{h^2}\right) < H^2, 
\end{equation}
where 
$$ A(X) := 1 - \frac{2 \pi^2}{9X} \quad \text{and} \quad B(X) := 1 + \frac{2 \pi^2}{9X} + \frac{1}{h} + \frac{\pi^2 \log X}{3hX}, $$
then $g(p)<H$.
\end{theorem}
This is~\cite[Theorem 3]{McGownTrudgian2020}, specialised to the case where $r=2$. 
\begin{proof}[Proof of Proposition~\ref{P:C=1.4195}]
Consider an arbitrary constant $C>0$. It is sufficient to prove Theorem~\ref{T: M and T} is invocable with $H = 2C \left(\frac{2\delta + s - 1}{\delta}2^{\omega-s}\right)^2 p^{3/8}$, where $C = 1.4195$, since that implies $g(p)<H$. Choose $h = \lambda p^{1/4}$, where $0.8165\leq \lambda < 0.8166$. Note that we require $h$ to be an integer, so formally $\lambda$ is a function of $p$ but always bounded in the interval given. We have
\begin{align*}
    2 h \leq H \impliedby & 2 \lambda p^{1/4} \leq 2 C \left(\frac{2\delta + s - 1}{\delta}2^{\omega-s}\right)^2 p^{3/8} \\
    \iff & \frac{\lambda}{p^{1/8}}\left(\frac{2\delta + s - 1}{\delta}2^{\omega-s}\right)^{-2} \leq C
\end{align*} 
Since $\omega:= \omega(p-1)$ and $p$ odd, $\omega \geq 2$ unless $p = 2^m + 1$. It is known that $2^m + 1$ is prime only if $m$ is itself a power of 2. Denote by $F_n:=2^{2^n}+1$, the $n$-th Fermat number. The only known Fermat primes are $F_0=3, F_1=5, F_2=17,F_3=257$ and $F_4=65537$ with $F_n$ composite for $n\in[5,32]$. Since $F_{32}= 2^{2^{32}}+1> 10^{10^9}$, no such primes can appear in our interval $10^{47}<p<10^{56}$, so $\omega \geq 2$. Thus
\begin{align*}
\frac{2\delta + s - 1}{\delta}2^{\omega-s} \geq & \frac{2\delta + s - 1}{\delta}2^{2-s} \\
\geq & \text{ min}\left\{\frac{2\delta_0 - 1}{\delta_0}2^{2},\frac{2\delta_1}{\delta_1}2^{1},\frac{2\delta_2 + 1}{\delta_2}\right\} \\
= & \text{ min}\left\{4,4,2+\frac{1}{\delta_2}\right\},
\end{align*}
where $\delta_i$ is interpreted as $\delta$ for $s=i$. But $\delta_2 = 1 - \frac{1}{2}-\frac{1}{p_2}$ for $p_2$ the only other factor of $p-1$ so $\delta_2\leq 1/2$ and 
\begin{equation}
    \frac{2\delta + s - 1}{\delta}2^{\omega-s} \geq 4. \label{B:LowerSieve}
\end{equation}
Since $\frac{\lambda}{4^2p^{1/8}}\leq 6.9\times 10^{-8}$, we conclude that
$$ 6.9\times 10^{-8} \leq C \implies 2h \leq H. $$
Suppose $2H^2\geq hp$. Then $H \geq (hp/2)^{1/2}$, so
$$ g(p) \geq H \implies g(p) \geq (hp/2)^{1/2} > (\lambda p^{1/4}p/2)^{1/2} \geq 0.4082 p^{5/8}.$$
By \cite[Theorem 5]{McGownSorensen2025}, this cannot happen since $g(p)<p^{0.56}<0.4082 p^{5/8}$, with the latter inequality satisfied whenever $p>10^7$. Therefore, we must have $2H^2 < hp$ or we are done since $g(p)<H$ and we may choose any $C\geq 6.9\times 10^{-8}$. We estimate $X = H/h$ via the lower bound from Equation~\eqref{B:LowerSieve}.
\begin{align*}
X  \geq & \ 2C \left(\frac{2\delta + s - 1}{\delta}2^{\omega-s}\right)^2 p^{3/8}/(\lambda p^{1/4}) \\
\geq & \ \frac{2}{\lambda} C \left(\frac{2\delta + s - 1}{\delta}2^{\omega-s}\right)^2p^{1/8} \geq \frac{32}{\lambda} C p^{1/8}.
\end{align*}
Thus $X \geq C\cdot 2.94\times 10^{7}$ since $p>10^{47}$. We proceed by estimating the quantities $A(X)$ and $B(X)$ in Equation~\eqref{E: T3}.
\begin{align*}
B(X) < & \ 1 + 2.18\times 10^{-12} + \frac{7.46\times 10^{-8} + 2.44 \times 10^{-19}\log C}{C} \\
A(X) > & \ 1 - \frac{7.47\times 10^{-8}}{C}
\end{align*}
We may choose $C\leq 1.5$ since its only restriction thus far is $6.9\times 10^{-8} \leq C$. Then 
\begin{align*}
B(X)^3 \leq & \ 1 +1.50 \times 10^{-7} \\
A(X)^4 \geq & \ 1- 2.00 \times 10^{-7}
\end{align*}
We are now in the position to invoke Theorem~\ref{T: M and T}:
\begin{align*}
    & H^2 - \frac{\pi^2}{6}\frac{B(X)^3}{A(X)^4}\left(\frac{2\delta + s-1}{\delta}2^{\omega-s}\right)^{4}hp^{1/2}\left(3+2\frac{p^{1/2}}{h^2}\right) \\
    \geq \ & 4C^2 \left(\frac{2\delta + s - 1}{\delta}2^{\omega-s}\right)^4 p^{6/8} - (8.0596)\left(\frac{2\delta + s-1}{\delta}2^{\omega-s}\right)^{4}p^{6/8}\ \\
    > \ & 4 \left(\frac{2\delta + s - 1}{\delta}2^{\omega-s}\right)^4 p^{6/8} \left(C^2 - 2.0149 \right)
\end{align*}
As long as $C^2\geq 2.0149$, we have $g(p)< H$, as required.
\end{proof}
By rearranging, we obtain a criterion for checking Grosswald's conjecture.
\begin{cor} \label{C: crit}
Suppose $10^{47}<p<10^{56}$. Let $e$ be an even divisor of $p-1$, with the factors $p_1\dots p_s \mid p-1$ but $p_i \nmid e$. If $\delta := 1- \sum_{i=1}^sp_i^{-1}>0$, then
\begin{equation} \label{E: crit} 4220 \left(\frac{2\delta + s - 1}{\delta}2^{\omega-s}\right)^{16} < p \quad \text{implies} \quad g(p) < \sqrt{p}-2. \end{equation}
\end{cor}
\section{Proof of Theorem~\ref{T: Gross int}}
By~\cite[Corollary 2]{McGownTrudgian2020}, we may assume $p<10^{56}$. Since the 35-th primorial satisfies
$$q_{35}\# := \prod_{k=1}^{35}q_k > 1.49 \times 10^{57}, $$
for $q_k$ the $k$-th prime, we may assume $\omega(p-1) \leq 34$. We bound 
$$ \frac{2\delta + s - 1}{\delta}2^{\omega-s} < \frac{2\delta' + s - 1}{\delta'}2^{\omega-s}, $$
where $\delta' := 1-\sum_{i=\omega-s+1}^{\omega} 1/q_i$ and $\omega := \omega(p-1)$. This follows from the $i$-th smallest prime divisor of $p-1$ being bounded trivially below by the $i$-th smallest prime $q_i$. Thus, we need not know the precise factorisation of $p-1$ in order to apply Corollary~\ref{C: crit}. Since we may choose $0\leq s \leq \omega$ freely, we minimise $\frac{2\delta' + s - 1}{\delta'}2^{\omega-s}$ across $s$. Note that this sieve is increasing with respect to $\omega$. Thus, whenever $\omega(p-1) \leq 34$, we have
$$ \Delta_\omega:= 4220\min_{0\leq s \leq \omega(p-1)}\left\{\left(\frac{2\delta' + s - 1}{\delta'}2^{\omega(p-1)-s}\right)^{16}\right\} \leq 6.25 \times 10^{54}, $$
occurring when $s=31$. Since $q_{34}\#>10^{55}$, all primes with $\omega(p-1)=34$ satisfy Grosswald. We repeat the procedure with $\omega(p-1)\leq 33$. This gives
$$ \Delta_{\omega} \leq 1.77 \times 10^{54}. $$ 
By Corollary~\ref{C: crit}, $p>1.77 \times 10^{54}$ implies $g(p) < \sqrt{p}-2$, as required.
\qed 
\vspace{8pt}

This proof outlines a general approach for reducing the interval for which Grosswald is unknown further. For a fixed $k$, find any primes $p$ satisfying $\omega(p-1) = k$ in the range
$$ q_{k}\# + 1 \leq p <\Delta_{k}. $$
Whenever $12\leq \omega \leq 33$, we have $q_\omega\#+1$ composite, so the lower bound can be improved slightly. For any primes that survive the filter, apply Corollary~\ref{C: crit} with the specific factorisation of $p-1$. This amounts to either $\delta$ increasing or being able to choose a better value for $s$ in the sieve. Any remaining values are undetectable by the sieve, and one verifies directly that $g(p)<\sqrt{p}-2$. Computationally, one constructs all $n$ below $\Delta_{k}$ with $\omega(n)=k$. For each $n$, Corollary~\ref{C: crit} is applied with the specific factorisation of $n$. If $n$ passes the criterion~\eqref{E: crit}, the primality of $n+1$ is checked.  


\begin{thebibliography}{9}
\bibliographystyle{plain}
\bibitem{Burgess1962}
D. A. Burgess,
\textit{On character sums and primitive roots},
Proc. London Math. Soc. (3) \textbf{12} (1962), 179--192.
\bibitem{Grosswald}
E. Grosswald,
\textit{On Burgess’ bound for primitive roots modulo primes and an application
to $\Gamma(p)$}, 
Amer. J. Math., (6) \textbf{103} (1981), 1171--1183.
\bibitem{McGownSorensen2025}
K. J. McGown and J. P. Sorenson, 
\textit{Computation of the Least Primitive Root}, 
Math. Comp. \textbf{94} (2025), 909–-917. 
\bibitem{McGownTrudgian2020}
K. J. McGown and T. Trudgian,
\textit{Explicit upper bounds on the least primitive root},
Proc. Amer. Math. Soc. (3) \textbf{148} (2020), 1049--1061.
\bibitem{McGownTrevinoTrudgian2017}
K.~McGown, E.~Trevi\~{n}o, and T.~Trudgian,
\textit{Resolving Grosswald’s conjecture on GRH},
Funct. Approx. (55) \textbf{2} (2016), 215-225.


\end{thebibliography}
\end{document}